\documentclass[11pt]{article}

\usepackage{amsmath, amsthm, amssymb, amsfonts}
\usepackage{geometry} 
\usepackage{hyperref}
\usepackage{mathrsfs}
\usepackage{extpfeil} 
\usepackage{enumitem}
\usepackage{titlesec}
\title{Partially Hyperbolic Dynamics on $\mathbb T^4$: Existence of Compact Center-Unstable Leaves }
\author{
Raul Ures, Tongyao Yu \\
SUSTech International Center for Mathematics \\
Department of Mathematics, Southern University of Science and Technology \\
Shenzhen, Guangdong, China
}
\date{\today}

\theoremstyle{plain}
\newtheorem{theorem}{Theorem}[section]
\newtheorem{lemma}[theorem]{Lemma}
\newtheorem{corollary}[theorem]{Corollary}
\newtheorem{proposition}[theorem]{Proposition}

\theoremstyle{definition}
\newtheorem{definition}[theorem]{Definition}

\theoremstyle{remark}
\newtheorem{remark}[theorem]{Remark}

\newcommand{\R}{\mathbb{R}} 
\newcommand{\Z}{\mathbb{Z}} 
\newcommand{\N}{\mathbb{N}} 
\titleformat{\section}{\centering\Large\bfseries}{\thesection.}{0.5em}{}
\begin{document}

\maketitle

\begin{abstract}
We show that for $n\ge2$,  if a partially hyperbolic diffeomorphism $f:\mathbb T^{n+1}\to \mathbb T^{n+1}$ with $\dim E^s=\dim E^c=1$  has an invariant center-unstable foliation with a compact incompressible leaf, then this foliation has a transverse closed curve in the universal cover. Also, if $f$  is leaf conjugate to its linear part, it has no compact incompressible center-unstable submanifold. In particular, by the incompressibility result we obtained on Anosov tori,  the incompressibility assumptions can be removed when $f$ is defined on $\mathbb T^4$.   
\end{abstract}

\section{Introduction}
This article deals with the integrability of invariant distributions in partially hyperbolic systems and attempts to generalize the result in \cite{RodriguezHertz2016}, which works on 3-manifolds, to higher dimensions. 

Firstly we consider the case of hyperbolic systems. A diffeomorphism $f:M\to M$ on a compact manifold $M$ is called Anosov if there is an invariant splitting $TM=E^s\oplus E^u$ such that, for some Riemannian metric $\|\cdot \|$ and for any $p\in M$ and unit vectors $v_s\in E^c_p$ and $v_u\in E^u_p$, we have 

$$
||df_p(v_s)\|<1<\|df_p(v_u)\|.
$$

Moreover, $f$ is called codimension one Anosov if one of $E^s$  and $E^u$  is one-dimensional. 

For an Anosov diffeomorphism $f: M\to M$, in light of Hadamard \cite{Hadamard1901} or Perron’s \cite{Perron1929} method, each of $E^s$ and $E^u$ is uniquely integrable. Partially hyperbolic systems have a splitting of the tangent bundle into three subbundles: stable, center, and unstable. The stable and unstable subbundles are uniformly contracted and expanded, respectively, while the center subbundle has intermediate behavior. By Hirsch, Pugh, and Shub and by Brin and Pesin, partially hyperbolic dynamics is introduced into differential dynamics to study systems that are not uniformly hyperbolic but still exhibit some form of robustness and regularity. For example, partially hyperbolic systems can arise from deformations of hyperbolic systems, skew products, hyperbolic flows and actions, iterated function systems, and attractors. Partially hyperbolic systems often have rich geometric and dynamical properties, such as invariant manifolds, dynamical coherence, robust transitivity, and finiteness of attractors. The precise definition is as follows:
\begin{definition}
    Let $M$ be a closed, connected Riemannian manifold. A diffeomorphism $f:M\to M$ is called partially hyperbolic if $TM$ has a non-trivial $df$-invariant splitting, denoted as $TM=E^s\oplus E^c\oplus E^u$ such that for some Remannian metric $||\cdot ||_{\cdot }$ of $M$, we have $||{df}|_{E^s}||<1,||{df^{-1}}|_{E^u}||<1$ and for any $p\in M$ and unit vectors $v^\sigma \in E^\sigma _p,(\sigma =s,c,u)$, we have $||df_pv^s||_{p}<||df_pv^c||_p<||df_pv^u||_p$. 

Also, we call $f$ a weakly partially hyperbolic diffeomorphism if exactly one of $E^s$ and $E^u$ is trivial.
\end{definition}
Still by Hadamard or Perron’s method, the strong distributions $E^s$ and $E^u$ are uniquely integrable. However, the distributions $E^{cs}:=E^c\oplus E^s, E^{cu}:=E^c\oplus E^u, E^c$ might not be integrable because of the intermediate behavior of the center distribution. It was A. Wilkinson ( see \cite{Wilkinson1998}) who first observed an algebraic Anosov diffeomorphism in a six-dimensional manifold presented in \cite{Smale1967}, where the subbundles are $C^\infty$-smooth, and the center bundle $E^c$ is $4$-dimensional and it does not satisfy the Frobenius condition. 

We say that a partially hyperbolic diffeomorphism $f: M\to M$ is $cs$-dynamically coherent if there is an $f$-invariant foliation tangent to $E^{cs}=E^c\oplus E^s$, and $cu$-dynamical coherence is well-defined analogously. Furthermore, we say that $f$ is dynamically coherent if it is both center-stable and center-unstable dynamically coherent. Then by taking intersections, there must be an $f$-invariant foliation tangent to $E^c$. In \cite{Burns2008}, there is a table illustrating the logical relationships between properties closely related to dynamical coherence.

Dynamical coherence is a key assumption in the theory of stable ergodicity. Also, determining whether or not a system is dynamically coherent is often an important first step towards understanding and classifying its dynamical behavior.

Actually, dynamical incoherence is very common among partially hyperbolic diffeomorphisms. Even in dimension 3, where the center bundle $E^c$ is one-dimensional hence the Frobenius condition is always trivially satisfied, although for decades, it remained an open question whether a partially hyperbolic system with a one-dimensional center was necessarily dynamically coherent, in \cite{RodriguezHertzToAppear} F. Rodriguez Hertz, J. Rodriguez Hertz, and R. Ures gave a non-dynamically coherent example on $\mathbb T^3$, revealing that the lack of Frobenius condition is not the only reason for the non-integrability of the center bundle. The problem comes from the lack of regularity. This discovery refreshed the understanding of dynamical behaviors that a partially hyperbolic system can have and prompted further studies into the role of invariant submanifolds. 

Also in \cite{RodriguezHertzToAppear}, the authors conjectured that any non-dynamically coherent partially hyperbolic diffeomorphism on a 3-manifold admits a periodic torus tangent to either $E^{cu}$ or $E^{sc}$. This conjecture was later confirmed for the 3-torus in \cite{Potrie2015}. Moreover, in \cite{HammerlindlPotrie2019}, the authors extended the result to 3-manifolds with virtually solvable fundamental groups (i.e., those containing a solvable subgroup of finite index). However, in \cite{Bonatti2020}, the authors found a $C^1$-open set of partially hyperbolic diffeomorphisms which are both transitive and dynamically incoherent, as an answer to that conjecture. The work in \cite{Bonatti2020} uses a technique called ‘$h$-transversality’ to build examples on the unit tangent bundles of surfaces. 

There are some sufficient conditions that imply dynamical coherence. We say that a foliation $\mathcal F$ of a manifold $M$ is quasi-isometric if, for the lifted foliation, $\tilde {\mathcal{F}}$ in the universal cover $\tilde M$: 
    \begin{align*}
       \exists a,b>0,s.t, \forall x,y\in L\in \tilde {\mathcal{F}},d_{\tilde {\mathcal{F}}}(x,y)\le ad(x,y)+b.
    \end{align*}
In \cite{Brin2003}, it is shown that on compact Riemannian manifolds, quasi-isometry of strong foliations implies locally unique integrability of the center-stable and center-unstable distributions, hence the system is dynamically coherent.

There is a stronger type of partial hyperbolicity, that is, absolute partial hyperbolicity:
\begin{definition}
    Let $M$ be a closed, connected Riemannian manifold. A diffeomorphism $f:M\to M$ is absolutely partially hyperbolic if it is partially hyperbolic and $\forall p,q,r\in M,v^s\in E^s_p,v^c\in E^c_q,v^u\in E^u_r$, we have $||df_pv^s||_{p}<||df_qv^c||_q<||df_rv^u||_r$.
\end{definition}
In \cite{Brin2009}, M. Brin, D. Burago, and S. Ivanov showed that on $\mathbb T^3$, absolutely partially hyperbolic diffeomorphisms are dynamically coherent, exactly by showing that the strong foliations are quasi-isometric. 

\subsection{Statement of Results}
\subsubsection{In Lower Dimensions}
Let $f:\mathbb T^n\to \mathbb T^n$ be a partially hyperbolic diffeomorphism with $n=3,4$ such that $\dim E^s=\dim E^c=1$, we have the following results:
\begin{theorem}
    \label{MainTheoremDim4}

If there is an $f$-invariant $cu$-foliation $\mathcal W^{cu}$ having a compact leaf $T$,  there exists a closed curve transverse to the lift $\tilde  {\mathcal{W}}^{cu}$. Hence if the leaves of $\tilde  {\mathcal{W}}^{cu}$ are simply connected, $\mathcal W^{cu}$ has no compact leaves.
\end{theorem}
\begin{remark}
    Note that when $n=3$, it is the main result in \cite{RodriguezHertz2016}
\end{remark}
\begin{theorem}
    \label{SecondTheoremCUDim4}

If $f$ is leaf conjugate to its linear part with respect to a center foliation $\mathcal W^c$ and suppose that there is a compact $cu$-submanifold $T$ of $f$. Then for any $x\in T$, $\mathcal W^c(x)$ is not contained in $T$. 
\end{theorem}

\begin{corollary}
    \label{Coro2SecondTheoremCUDim4}

If $f$ is leaf conjugate to its linear part with respect to a center foliation $\mathcal W^c$, then $f$ has no locally uniquely integrable $cu$-submanifolds.
\end{corollary}
The results above are actually consequences of the following more general results that hold in higher dimensions:
\subsubsection{In Higher Dimensions}
For $n\ge 2$, let $f:\mathbb T^{n+1}\to \mathbb T^{n+1}$ be a partially hyperbolic diffeomorphism such that $\dim E^s=\dim E^c=1$.
\begin{theorem}
    \label{MainTheorem}
If there is an $f$-invariant $cu$-foliation $\mathcal W^{cu}$ having a compact incompressible leaf $T$, there exists a closed curve transverse to the lift $\tilde  {\mathcal{W}}^{cu}$. Hence if leaves of $\tilde  {\mathcal{W}}^{cu}$ are simply connected, $\mathcal W^{cu}$ has no compact incompressible leaves.
\begin{theorem}
    \label{SecondTheoremCU}
If $f$ is leaf conjugate to its linear part with respect to a center foliation $\mathcal W^c$ and suppose that there is a compact incompressible $cu$-submanifold $T$ of $f$. Then for any $x\in T$, $\mathcal W^c(x)$ is not contained in $T$.

\end{theorem}

\end{theorem}
\begin{remark}
    Corollary \ref{Coro2SecondTheoremCUDim4} still holds for incompressible $cu$-submanifolds in the general setting in higher dimensions. 
\end{remark}

\begin{definition}
    
Let $T$ be a compact embedded submanifold of $\mathbb T^{n+1}$, we say that $T$ is an Anosov Torus if there is a diffeomorphism $f:M\to M$ such that $f(T)=T$ and $(f|_T)_*:\pi_1(T)\to \pi_1(T)$ induces a hyperbolic integer matrix.

\end{definition}

The following theorem on the incompressibility of Anosov tori lays the geometric foundation for obtaining the results in dimension $4$ from the results above:
\begin{theorem}
    \label{IrreAnosovTori}
Let $T$ be an irreducible Anosov $n$-tori, in particular, codimension one Anosov $n$-tori, in a torsion-free Riemannian manifold $M$ covered by $\mathbb R^{n+1}$. One of the following holds:
\begin{enumerate}[label=(\alph*)]
    \item Either $T$ is incompressible, or
    \item for $n\ge3$, $T$ is totally compressible and can be embedded into $\mathbb S^{n+1}$ by taking a lift in $\mathbb R^{n+1}$. Moreover, denote $\tilde T$ as such an embedding in $\mathbb S^{n+1}$, then $\tilde T$ splits $\mathbb S^{n+1}$ into two components, $N_1$ and $N_2$, and one of them, say, $N_1$ is simply connected and $N_2$ has an incompressible boundary, also, $H_1(N_2)\simeq H_1(\partial N_2)$. Furthermore, let $G:= [\pi_1(N_2),\pi_1(N_2)]\rtimes \pi_1(\partial N_2)$, where $\pi_1(\partial N_2)$ acts on $[\pi_1(N_2),\pi_1(N_2)]$ by conjugation (  $\pi_1(\partial N_2)$ is identified with its image under the inclusion map ), then $G\simeq \pi_1(N_2)$ and $\pi_1(\partial N_2)$ normally generates $G$.
\end{enumerate}
\end{theorem}

\subsection{Organization of the Paper}
The paper is organized as follows: In Section \ref{CodimensionOnePHD} we introduce codimension one partially hyperbolic systems and verify several lemmas with respect to Frank’s semiconjugacy as higher-dimensional analogues of those in \cite{RodriguezHertz2016}. Section \ref{CUsub} is about center-unstable submanifolds and a corollary of Theorem \ref{IrreAnosovTori} will be proved in this section, which shows that Theorem \ref{MainTheoremDim4} and \ref{SecondTheoremCUDim4} follow from Theorem \ref{MainTheorem} and  \ref{SecondTheoremCU}, respectively. Theorem \ref{SecondTheoremCU} will also be proved in Section \ref{CUsub}. Then we give the proof of Theorem \ref{MainTheorem} in Section \ref{NeoCLv}. Finally, as a relatively independent result, we place the proof of Theorem \ref{IrreAnosovTori} in the last section. 
\newpage
\section{Codimension One Partial Hyperbolicity}
\label{CodimensionOnePHD}

There is a family of diffeomorphisms that are of both a specialization of weakly partially hyperbolicity and a generalization of codimension one hyperbolicity: 

Let $M$ be a closed Riemannian manifold. A $C^1$-diffeomorphism $f: M \to M$ is called a codimension one partially hyperbolic diffeomorphism, if $f$ or $f^{-1}$ admits a continuous $T f$-invariant splitting $TM = E^c \oplus E^u$ and $E^c$ is one-dimensional, and a function $\xi : M \to (1,+\infty)$ such that$\|T f(v^c)\| < \xi(x)< \|T f(v^u)\|$ for all $x \in M$ and unit vectors $v^c \in E^c_x$ and $v^u \in E^u_x$.

By the results in \cite{Newhouse1970}, tori are the only type of closed and connected manifold supporting a codimension one Anosov diffeomorphism, and all codimension one Anosov diffeomorphisms are topologically conjugate to a toral automorphism. Furthermore, in \cite{Zhang2023}, a generalized result on codimension one partial hyperbolicity is obtained:
\begin{theorem}( Theorem 1.2 in \cite{Zhang2023})
    \label{ZXCodimension1PHD}
    Let $f: M\to M$ be a codimension one partially hyperbolic diffeomorphism on a closed, connected manifold $M$, with a splitting $TM=E^c\oplus E^u$. Then:
    \begin{itemize}
        \item  $N$ is homeomorphic to $\mathbb T^n$
        \item There is a foliation ${\mathcal W}^{c}$ tangent to $E^c$ such that any piecewise $C^1$-curve tangent to $E^c$ must lie in a unique leaf of ${\mathcal W}^{c}$
        \item  $f$ is semiconjugated to its linear part $f_*$, which is codimension one Anosov.
    \end{itemize}
\end{theorem}
\begin{remark}
    \label{OneDimensionalStable}

Denote $A$ as the hyperbolic $n\times n$ integer matrix induced from $f_*$. By Proposition 2.18 of \cite{Zhang2023}, one can take a unique codimension one vector subspace $W$ in $\mathbb R^n$ and $R>0$, such that each leaf of $\tilde{\mathcal W^u}$ lies in the $R$-neighborhood of a translate of $W$ and $\forall x\in \mathbb R^n$, the $R$-neighborhood of the leaf $\tilde{\mathcal W^u}(x)$ contains $W+x$. By Proposition 4.3 of \cite{Zhang2023}, $W$ is actually spanned by eigenvectors whose eigenvalues have absolute values greater than 1. So the stable distribution of $A$ has only one dimension.
\end{remark}
Now let $f:M\to M$ be a codimension one partially hyperbolic diffeomorphism on a closed $n$-manifold $M$ with splitting $TM=E^c\oplus E^u$. By Theorem \ref{ZXCodimension1PHD}, $M\simeq \mathbb T^n$. By remark \ref{OneDimensionalStable}, $f_*:\pi_1(M)\to \pi_1(M)$ induces an $n\times n$ hyperbolic integer matrix $A$ with exactly one-dimensional stable distribution. Also, there is a foliation ${\mathcal W}^{c}$ tangent to $E^c$ such that any piecewise $C^1$-curve tangent to $E^c$ must lie in a unique leaf of ${\mathcal W}^{c}$, hence $E^c$ is uniquely integrable. Then ${\mathcal W}^{c}$ is the unique foliation tangent to $E^c$ and is invariant under $f$. 
\begin{lemma}
\label{CUGPS}
    There is a global product structure between $\tilde {\mathcal W}^{c} $ and $\tilde {\mathcal W}^u$ in the universal cover, and both of $\tilde {\mathcal W}^{c} $ and $\tilde {\mathcal W}^u$ are quasi-isometric. 
    \begin{proof}
        Since $\mathcal W^u$ is codimension one and is without holonomy and ${\mathcal W}^{c} $ is transverse to it, Theorem VIII. 2.2.1  in \cite{Hector1983} implies that there is a global product structure between $\tilde {\mathcal W}^{c} $ and $\tilde {\mathcal W}^u$. Furthermore, by Proposition 2.24 in \cite{Zhang2023}, the global product structure implies the quasi-isometry of $\tilde {\mathcal W}^{c} $. Also, $\mathcal W^u$ is quasi-isometric in light of Proposition 2.22 in \cite{Zhang2023}
    \end{proof}
\end{lemma}

Denote $\pi:{\mathbb {R}}^n\to \mathbb T^n$ as the universal covering map and let $\tilde f :{\mathbb {R}}^n\to {\mathbb {R}}^n$ be a lift of $f$. In light of the shadowing property of $A$, there is $K_1>0$ such that $\forall \tilde x\in \mathbb R^n$, there is a unique $\tilde y\in \mathbb R^n$ such that $\| {\tilde f}^{n}(\tilde x)-A^{n}(\tilde y) \|\le K_1,\forall n\in \mathbb Z$. Then we have the well-defined continuous map $H:\tilde x\mapsto \tilde y$ at a distance not larger that $K_1$ from the identity, such that $H\circ A=\tilde f\circ  H$. Then it is surjective ( This follows from a degree argument, see Chapter 2 of Hatcher’s book, \cite{Hatcher2002}  ) and $\sup_{\tilde x\in \mathbb R^n}\text{diam }H^{-1}(\tilde x)<\infty$. 

Also, it is not difficult to see that $H$ is $\mathbb Z^n$-periodic so $H$ induces a surjective $h_f:\mathbb T^n\to \mathbb T^n$, such that $H$ is a lift of $h_f$ and $A\circ h_f=h_f\circ f$.
Also, one can directly obtain $h_f$ from $f$ by results in \cite{Franks1970} and $H$ is a lift of $h_f$ and $h_f$ is homotopic to $\text{Id}_{\mathbb T^3}$, since $A$ is a $\pi_1$-diffeomorphism.

\begin{lemma}
    \label{hgmapsUtoU}
    $h_f$ injectively maps unstable manifolds of $f$ to unstable manifolds of $A$.
    \begin{proof}
It is equivalent to consider the universal cover. 

Firstly, for any two distinct points $\tilde x,\tilde{y}$ in  $\mathbb R^n$  in the same unstable manifold of $\tilde f$. Note that by Lemma \ref{CUGPS}, $\tilde {\mathcal W}^u$ is quasi-isometric, so there are $a,b>0$ such that $\forall n\in \N_+, d_{\tilde {\mathcal W}^u}({\tilde f^n}(x),{\tilde{f}^n}(y))\le ad({\tilde f^n}(x),{\tilde{f}^n}(y))+b$. Hence $d({\tilde f^n}(x),{\tilde{f}^n}(y))\to \infty$ as $n\to \infty$. 

Also, notice that $\tilde f({H}^{-1}(\tilde x))={H}^{-1}(A(\tilde x))$ and $\sup_{\tilde x\in \mathbb R^n}\text{diam }H^{-1}(\tilde x)<\infty$, hence any two distinct points in  $\mathbb R^n$ in the same unstable manifold of $\tilde f$ cannot lie in the same fiber of $H$. This implies the injectivity of $h_f$ on unstable manifolds of $f$. 

Finally, the diameter of the right-hand side of the identity $A^n\circ H(\{\tilde x,\tilde y\})=H\circ {\tilde f}^n(\{\tilde x,\tilde y\})$ goes to $0$ as $n\to -\infty$. So the diameter of $A^n\circ H(\{\tilde x,\tilde y\})$ remain uniformly bounded as $n\to -\infty$.  Thus $H(\tilde{x})$ and $H(\tilde{y})$ must lie in the same unstable leave of $A$.
    \end{proof}
\end{lemma}
\begin{lemma}
    \label{hgmapsCtoS}

$h_f$ maps center curves into stable manifolds of $A$
\begin{proof}
 Take a center arc $\gamma$  of $\mathbb T^n$ small enough. For $\delta >0,$ denote $\mathcal W^u_\delta(\gamma ) :=\cup _{z\in \gamma  }\mathcal W^u_\delta(z)$. And we take $\delta$  small enough such that for any distinct $x,y\in \gamma$  we have that $\mathcal W^u_\delta(x)\cap \mathcal W^u_\delta(y)=\emptyset$. Obviously $\forall m\in \N_+$, $f^m(\mathcal W^u_\delta(\gamma ) )\supset \mathcal W^u_\delta(f^m(\gamma ) )$  thus by a ‘volume versus length’ argument ( note that the unstable distribution is codimension one in $\mathbb T^n$ ) implies that for some $C>0$: 

$$
\forall m\in \N,\infty>\text{Vol}(\mathbb T^n)\ge \text{Vol}(f^m(\mathcal W^u_\delta(\gamma ) ))\ge \text{Vol}(\mathcal W^u_\delta(f^m(\gamma ) ) )\ge C\cdot \text{length}(f^m(\gamma ))
$$
Therefore the length of $\{f^m(\gamma )\}_{m\in \N} $ is uniformly bounded, hence $\{H\circ {\tilde f}^m(\tilde \gamma) \}_{m\in \N} =\{A^m\circ H(\tilde \gamma )\}_{m\in \N}$ has a bounded diameter, where $\tilde \gamma \subset \mathbb R^n$ is a lift of $\gamma $. This fact obviously forces $H(\tilde \gamma )$ to lie in a stable manifold of $A$, and the proof is finished after mapping down by the covering map. 
\end{proof}
\end{lemma}

\begin{proposition}
    \label{fibersareCarcs}

$\forall z\in \mathbb T^n,$ ${h_f}^{-1}(z)$ is a center arc in $\mathbb T^n$ ( might be a trivial arc ).
\begin{proof}
Firstly we show that ${h_f}^{-1}(z)$ is contained in a center leaf of $H$. If not, there are $\tilde x,\tilde y\in \mathbb R^n$ in different center leaves with respect to $H$ such that $H(\tilde x)=H(\tilde y)$, then, by the global product structure shown in Lemma \ref{CUGPS} , there is $\tilde z \ne \tilde x$, the intersection point between the unstable manifold of $\tilde x$ and the center manifold of $\tilde y$ with respect to $\tilde f$. Then, in the image of $H$, by Lemma \ref{hgmapsUtoU}, \ref{hgmapsCtoS}, an unstable leaf of $A$ has to intersect a stable leaf of $A$ twice, which is impossible. 

Also, we need to show that for each $z\in \mathbb T^n$, ${h_f}^{-1}(z)$ is connected. If not, there is $\tilde z\in \mathbb R^n$ such that $H^{-1}(\tilde z)$ is not connected. By the continuity of $H$, there is no isolated points in $\tilde {\mathcal W}^c(\tilde x)\setminus H^{-1}(\tilde z)$, where $\tilde x\in H^{-1}(\tilde z)$. 
Further, since the center curves are properly embedded by their quasi-isometry shown in Lemma \ref{CUGPS}, $H^{-1}(\tilde z)$ is compact in $\tilde {\mathcal W}^c(\tilde x)$ with at least two components. Hence there must be an open interval $I \subset \tilde {\mathcal W}^c(\tilde x)\setminus H^{-1}(\tilde z)$ between the elements of $H^{-1}(\tilde z)$ in $\tilde {\mathcal W}^c(\tilde x)$ such that $\overline I \cap H^{-1}(\tilde z) \ne \emptyset$
Then since the diameter of $\{H ^{-1}(A^m(\tilde z))\}_{m\in \mathbb Z}$ are uniformly bounded, and for any $m\in \mathbb Z,{\tilde f}^m({H}^{-1}(\tilde z))={H}^{-1}(A^m(\tilde z))$, we have that $\{{\tilde f}^m({H}^{-1}(\tilde z))\}_{m\in \mathbb Z}$ 
is uniformly bounded with respect to the arc length of $\tilde {\mathcal W}^c$ by its quasi-isometry. As a result, $\{{\tilde f}^m(\overline {I})\}_{m\in \mathbb Z}$ is also uniformly bounded with respect to the arc length, but it is not possible since $H(\overline {I})$ consists of different points and we have that for all $m\in \mathbb Z,H\circ {\tilde f}^m=A^m\circ H$.

\end{proof}
\end{proposition}
\begin{lemma}
    \label{hgfibersareC}

For all $x\in \mathbb T^n$, there must be some $z\in  {{\mathcal W}}^s_A(h_f(x))$ such that ${h_f}^{-1}(z)$ consists of one point. 
\begin{proof}
  Let $\tilde z_1,\tilde z_2\in \mathbb R^n$ be two distinct points such that $\tilde z_2\in {\tilde {\mathcal W}^s}_A(\tilde z_1)$ and suppose that $H^{-1}(\tilde z_2)$ and $H^{-1}(\tilde z_1)$ are not in the same center leaf of $\tilde f$. By the global product structure between $\tilde{\mathcal W}^c$ and $\tilde {\mathcal W}^u$ , there are $\tilde x\in H^{-1}(\tilde z_1)$ and $\tilde y\in H^{-1}(\tilde z_2)$. Let $\tilde z$ be the intersection point between the unstable manifold of $\tilde x$ and the center manifold of $\tilde y$ with respect to $\tilde g$. Then in the image of $H$, an unstable leave of $A$ has to intersect a stable leave of $A$ twice, which is impossible. Hence $\forall y\in { {\mathcal W}}^s_A(h_f(x))$, $h_f^{-1}(y)\subset \mathcal W^c(x)$. 
Also, by Lemma \ref{hgmapsCtoS}, we have $h_f(\mathcal W^c(x))\subset { {\mathcal W}}^s_A(h_f(x))$, as a result, we have

\begin{align*}
    \mathcal W^c(x)\subset h_f^{-1}\circ h_f(\mathcal W^c(x))\subset h_f^{-1}({ {\mathcal W}}^s_A(h_f(x)))\\=\cup _{y\in { {\mathcal W}}^s_A(h_f(x))}h_f^{-1}(y)\subset \cup_{y\in { {\mathcal W}}^s_A(h_f(x))}\mathcal W^c(x)=\mathcal W^c(x).
\end{align*}
Hence for all $x\in \mathbb T^n,h_f^{-1}({ {\mathcal W}}^s_A(h_f(x)))=\mathcal W^c(x)$, then points in $ {{\mathcal W}}^s_A(h_f(x))$ defines an uncountable partition of $\mathcal W^c(x)$, so there must be some $z\in  {{\mathcal W}}^s_A(h_f(x))$ such that ${h_f}^{-1}(z)$ consists of one point. ( Infact, this happens for all but countably many points )
\end{proof}
\end{lemma}
\begin{proposition}
    \label{periodicpointwithsmallfiber}

$\forall \epsilon>0$ there is a periodic point $p\in \mathbb T^n$ such that $\gamma _p:={h_f}^{-1}(h_f(p))$ has length smaller than $\epsilon$.
\begin{proof}
By Lemma \ref{hgfibersareC}, take $x\in T_0$ such that ${h_f}^{-1}(h_f(x))$ consists of one point. Then since $A$ is a hyperbolic toral automorphism, we can take a sequence of periodic points $q_m$ approximating $h_f(x)$. Denote $c_m$ as the length of ${h_f}^{-1}(h_f(q_m))$. We now show that $c_m\to 0$ as $m\to \infty$. 

If not, by passing to a subsequence we can always find $\delta >0$ and a pair $(x_m,y_m)\in \mathbb T^n\times \mathbb T^n$ such that $d(x_m,y_m)\ge \delta$  and $h_f(x_m)=h_f(y_m)=q_m$. 

Since $\mathbb T^n\times \mathbb T^n$ is compact, by passing to a subsequence again, there is $x',y'$ such that $x_m\to x'$ and $y_m\to y'$, then we have $d(x',y')\ge \delta$  but $h_f(x')=h_f(y')=h_f(x)$, leading to a contradiction. 

Now take $c_m<\epsilon$ and consider the arc ${h_f}^{-1}(q_m)$, so there is a period $k\in \N_+$ such that ${\tilde f}^k({h_f}^{-1}(q_m))={h_f}^{-1}(q_m)$. Then one can find a fixed point $p$ under ${\tilde{f}}^k$  since it is essentially a map of the interval. 
\end{proof}
\end{proposition}
\newpage
\section{Compact Center-Unstable Submanifolds}
\label{CUsub}
Let $M$ be a closed connected Riemannian manifold and $f:M\to M$ a partially hyperbolic diffeomorphism with the splitting $TM=E^s\oplus E^c\oplus E^u$. 
\begin{definition}
    An immersed submanifold $S\subset M$ is called a $cu$-submanifold if it is tangent to $E^{cu}:=E^c\oplus E^u$. We call $cu$-submanifold homeomorphic to an $n$-torus a $cu$-$n$-torus and we will directly say a $cu$-torus if $n=2$. 
\end{definition}
\subsection{Restrictions on the underlying manifolds}
When $M$ is a closed connected 3-manifold, note that compact $cu$-submanifolds are foliated by lines so all of them are tori by Poincaré-Hopf.

The existence of a $cu$-torus in a 3-manifold poses a strong restriction on the underlying manifold. In \cite{RodriguezHertz2016} the authors further developed the results in \cite{Hertz2011}:
\begin{theorem} ( Theorem 1.3 in \cite{RodriguezHertz2016} )
\label{Toriwith}
Let $f$ be a partially hyperbolic diffeomorphism of an orientable 3-manifold. If there exists either a $cu$-torus or an $su$-torus, then $M$ is a mapping torus of either $\pm\text{Id}$ or a hyperbolic toral automorphism.

\end{theorem}
Then by the classification result in \cite{HammerlindlPotrie2015}, we have the following corollary:
\begin{corollary}
    Any closed, connected, orientable 3-manifold with a virtually solvable fundamental group must be one of the listed in Theorem \ref{Toriwith}, if it admits a dynamically incoherent partially hyperbolic diffeomorphism. 
\end{corollary}
Further, a complete classification of partially hyperbolic diffeomorphisms of closed oriented 3-manifolds with a center-stable or center-unstable torus can be seen in \cite{HammerlindlPotrie2019}. For those diffeomorphisms, one can find a finite pairwise disjoint collection of all center-stable and center-unstable torus. Then those tori can cut $M$ into components in which an iterate of $f$ is conjugate to a ‘skew-product’ over a hyperbolic toral automorphism ( 2-dimensional ) and a continuous map. 

A higher dimensional analogue of the main result in \cite{RodriguezHertz2016} should be Theorem \ref{MainTheorem}, by the following two reasons: 

On the one hand, the main result in \cite{Hertz2011}, which is the theoretical foundation to extend the main result in \cite{RodriguezHertz2016} from $M=\mathbb T^3$ to all closed 3-manifolds, follows from the Jaco-Shalen-Johannson decomposition adapted to the dynamics setting, which is a powerful tool on 3-manifold topology and is not well-established in higher dimensions. 

On the other hand, the bundle-switching method originally introduced in \cite{RodriguezHertzToAppear} was developed in \cite{Hammerlindl2025}. The author claims that for any $d\in \N_+$, under some conditions related to $d$, one can construct a dynamically incoherent partially hyperbolic system on $N\times \mathbb T^d$ from a partially hyperbolic system on a closed manifold $N$, leaving $N\times \{0\}$ as a fixed $cu$-submanifold. It implies that for $n\ge 4$, the family of closed $n$-manifolds having a compact codimension $d$ $cu$-submanifold is not smaller than that of $n-d$ manifolds supporting a certain type of weakly partially hyperbolic diffeomorphism. 
\subsection{Periodicity, Hyperbolicity, and Incompressibility of Compact Center-unstable Submanifolds}
\begin{definition}
    For an immersed submanifold $S\subset M$, we say that it is periodic with respect to $f$ if there is $n\in \N_+$ such that $f^n(S)=S$. If $f(S)=S$, we say that $S$ is fixed or invariant. Also, we say that $S$ is incompressible if the induced fundamental group homomorphism $i_*:\pi_1(S)\to \pi_1(M)$ of the inclusion map $i:S\hookrightarrow  M$ is injective and we say that $S$ is totally compressible if $\ker i_*=\pi_1(S)$
\end{definition}
\begin{definition}
For an embedded $n$-torus $T$ in $M$, we say that $T$ is Anosov if there is a diffeomorphism $g:M\to M$ fixing $T$ such that $(g|_T)_*:\pi_1(T)\to \pi_1(T)$ induces a hyperbolic integer matrix. Furthermore, for an Anosov $n$-torus $T$ in $M$, we say that $T$ is irreducible if one can choose $g$ such that the characteristic polynomial of the induced matrix is irreducible. Then similarly, one can define codimension one Anosov $n$-tori. 
\end{definition}
In \cite{RodriguezHertz2016}, the authors showed that on 3-manifolds all compact $cu$-submanifolds are Anosov tori and then moreover, by Theorem 4.1 in \cite{RodriguezHertz2008}, are incompressible.

In \cite{RodriguezHertz2016}, starting from a $cu$-torus $T$, consider the backward iterates $f^{-n}(T)$. Since the space of compact subsets of $M$ is compact in the Hausdorff metric, a subsequence $f^{-n_k}(T)$ converges, so for large $N\gg L$ the tori $f^{-N}(T)$ and $f^{-L}(T)$ are very close. Using a small stable tubular neighborhood $U(T)$, this closeness implies a trapping relation for some $m=N-L$: $f^{m}(U(T))\subset U(T)$. Then $T_*=\bigcap_{j\ge 0} f^{jm}(U(T))$ is a compact $f^{m}$-invariant set, and in fact a periodic embedded torus diffeomorphic to $T$, still tangent to $E^{cu}$. Hence $T_*$ is a periodic $cu$-torus.

Then set $g=f^m$, for a period $m$ of $T_*$, by a Poincaré–Bendixson argument, the authors showed that $(g|_{T_*})_*:\pi_1(T_*)\to \pi_1(T_*)$ induces a $2\times 2$ hyperbolic integer matrix. Then combining the result in \cite{RodriguezHertz2008}, we have the following:
\begin{proposition}( Propositions 2.2, 2.3  in \cite{RodriguezHertz2016} )
    The existence of a $cu$-torus implies the existence of a periodic Anosov $cu$-torus.

\end{proposition}
In \cite{Hammerlindl2018}, the authors achieved a further understanding of compact $cu$-submanifolds and the result is generalized to arbitrary dimensions, for example, the ‘approximating’ shown in Theorem 3.1 actually does not occur:
\begin{theorem}( Theorems 2.1, 2.2, 2,3 in \cite{Hammerlindl2018} )
    \label{AndyCptCuSubmanifolds}
A partially hyperbolic diffeomorphism has at most a finite number of compact $cu$-submanifolds, and each of them are periodic. And for $S$ as a compact periodic $C^0$ submanifold, it is a $cu$-submanifold if and only if $\forall x\in S, \mathcal W^s(x)\cap S=\{x\}$.
\end{theorem}
\begin{corollary}
    For a partially hyperbolic diffeomorphism $f$ on a closed $3$-manifold, all compact $cu$-submanifolds are periodic Anosov 2-tori. 
\end{corollary}
The following result will be proved assuming Theorem \ref{IrreAnosovTori}. 
\begin{theorem}
\label{cptCUsubmanifoldincom}
Let $f:M\to M$ be a partially hyperbolic diffeomorphism on a closed manifold $M$ covered by $\mathbb R^{4}$, such that $\dim E^s=\dim E^c=1$. Then any compact $cu$-submanifold is a periodic incompressible codimension one Anosov $3$-torus. 
\begin{proof}
 Let $T$ be a compact $cu$-submanifold with respect to $f$, since it is compact and injectively immersed, it is embedded in $M$. Theorem \ref{AndyCptCuSubmanifolds} implies that it is periodic under $f$, then one can take a period $k\in \N_+$ such that $f^k(T)=T$. Let $g=f^k$, so $g|_T$ is a codimension one partially hyperbolic diffeomorphism. Also, Theorem \ref{ZXCodimension1PHD} implies that $T$ is a codimension one Anosov $3$-torus. 

Suppose that $T$ is not incompressible, in light of Theorem \ref{IrreAnosovTori}, which will be proved in the last section, we have the components $N_1$ and $N_2$ such that $H_1(N_1)=0,H_1(N_2)\simeq \mathbb Z^3$.
Consider the long exact sequence: 

$$
\cdot \cdot \cdot \xrightarrow{}H_1(N_1,\partial N_1)\xrightarrow{}H_0(\partial N_1)\xrightarrow{i_*}H_0(N_1)\xrightarrow{j_*}H_0(N_1,\partial N_1)\xrightarrow{}0
$$

By Lefschetz duality and Alexander duality, 

$$
H_1(N_1,\partial N_1)\simeq H^3(N_1)\simeq \tilde{H}^3(N_1)\simeq  \tilde{H}_0(N_2)=0
$$

So by exactness of the sequence, $i_*$ is injective and since $H_0(\partial N_1)\simeq H_0(N_1)=\mathbb Z$, and $i_*$ is the induced homology group homomorphism of the inclusion map $i:\partial N_1 \hookrightarrow N_1
$, $i_*$ is an isomorphism. Thus $\ker j_*=H_0(N_1)$, that is, $j_*$ is a trivial homomorphism. But $j_*$ has to be surjective as well, we have that $H_0(N_1,\partial N_1)=0$ and similarly, $H_0(N_2,\partial N_2)=0$. Then Lefschetz duality implies:

$$
H^4(N_1)\simeq H^4(N_2)=0
$$

Then consider the Euler characteristic, with similar computations, also by the universal coefficient theorem, we have

\begin{align*}
    \chi (N_1) &= \dim H_0(N_1) - \dim H_1(N_1) + \dim H_2(N_1) - \dim H_3(N_1) + \dim H_4(N_1) \\
    &= 1-0+\dim H^1(N_2)-\dim \tilde {H}^0(N_2)+0\\&=1+\dim H_1(N_2)-\dim \tilde{H}_0(N_2)\\
    &= 1+3-0=4
\end{align*}
and 

\begin{align*}
    \chi (N_2) &= \dim H_0(N_2) - \dim H_1(N_2) + \dim H_2(N_2) - \dim H_3(N_2) + \dim H_4(N_2) \\
    &= 1 - 3+ \dim H^1(N_1) - \dim \tilde{H}^0(N_1) + 0 \\
    &=-2+\dim H_1(N_1)-\dim \tilde{H}_0(N_1)\\&=-2
\end{align*}

In light of the Jordan-Brouwer Separation Theorem and note that we identify ${\mathbb S}^{4}$ with $\mathbb R^{4}\sqcup\{\infty\}$, one of $N_1$ and $N_2$ is a compact subset of $\mathbb R^{4}$, denoted as $N_*$.

Now consider the stable distribution $\tilde E^s$ of a lift $\tilde g:\mathbb R^{4}\to \mathbb R^{4}$ of $g$. Since $T$ is a $cu$-submanifold, we have a non-vanishing continuous vector field $\tilde X^s$ defined on $N_*$, in the direction of $\tilde E^s$, pointing in the outward normal direction along $\partial N_*$. Then one can take a smooth non-vanishing vector field $\tilde X$ on $N_*$ approximating $\tilde X^s$,  still pointing in the outward normal direction along $\partial N_*$. Hence the Poincaré–Hopf Theorem shows that $N_*$ has zero Euler characteristic. It is impossible since $\chi (N_1)=4,\chi(N_2)=-2$ and $N_*$ is one of them.
\end{proof}
\end{theorem}
By Proposition \ref{cptCUsubmanifoldincom}, Theorem \ref{SecondTheoremCUDim4} immediately follows from Theorem \ref{SecondTheoremCU}, of which we will give a proof right now:
\begin{proof}[Proof of Theorem \ref{SecondTheoremCU}]
    
Now suppose $f$ has a compact incompressible $cu$-submanifold $T$. In light of Theorem  \ref{AndyCptCuSubmanifolds} and taking a finite iterate, we might as well assume that $T$ is $f$-invatiant.  Also, by Theorem  \ref{ZXCodimension1PHD}, $T$ is a codimension one Anosov $n$-torus. 

Then denote the induced $(n+1)\times (n+1)$ integer matrix of $f_*:\pi_1(\mathbb T^{n+1})\to \pi_1(\mathbb T^{n+1})$ as $B$, we have $B=A\times I$, where $A$ is the $n\times n$ hyperbolic integer matrix induced from $(f|_T)_*:\pi_1(T)\to \pi_1(T)$, which is with exactly one-dimensional stable distribution, by Remark \ref{OneDimensionalStable}.

Since $f$ is leaf conjugate to $N$, there is a homeomorphism $\psi :\mathbb T^{n+1}\to \mathbb T^{n+1}$, such that $\forall \mathcal L\in {\mathcal W}^{c},\psi\circ f(\mathcal{L})=N\circ \psi (\mathcal{L})$. Hence $\forall \mathcal L\in {\mathcal W}^{c},f(\mathcal{L})=\psi ^{-1}\circ N\circ \psi (\mathcal{L})=\psi ^{-1}\circ\psi( \mathcal{L})=\mathcal{L}$ and $\mathcal{L}$ is compact. 

Suppose that there is $x\in T$ such that $\mathcal W^c(x)\subset T$ then denote $\tilde {\mathcal L}$ as a lift of it in the universal cover $\tilde T\simeq \mathbb R^n$ of $T$. Then take $\tilde f_T :\mathbb{R}^n\to \mathbb{R}^n$ as a lift of $f|_T$. By Lemma \ref{hgfibersareC} we have $H(\tilde{\mathcal{L}})=H\circ H^{-1}(\tilde{\mathcal{W}^{s}_A}(H(\tilde{x})))=\tilde{\mathcal{W}^{s}_A}(H(\tilde{x}))$. However, $\tilde{\mathcal{W}^{s}_A}(H(\tilde{x}))$ is a line with irrational slope and $\tilde{\mathcal{L}}$ is a line with rational slope by the compactness of $\mathcal{L}$ and this contradicts to the fact that $H$ has a bounded distance from the identity.
\end{proof}
\begin{remark}
    The proof relies on the fact that $H$ cannot be a map between two non-parallel lines whereas if all center leaves in the ambient manifold $\mathbb T^4$ are not contained in $T$ , the contradiction could be avoided. Actually in light of Theorem  1.3 of \cite{HammerlindlPotrie2014}, partially hyperbolic diffeomorphisms must be leaf conjugate to their linear parts if it is dynamically coherent. And there is a good example in \cite{RodriguezHertzToAppear} ( See Lemma 2.5(2) of \cite{RodriguezHertzToAppear} ) on $\mathbb T^3$.
\end{remark}
\newpage
\section{Existence of Compact Leaves}
\label{NeoCLv}
In this section, we only need to prove Theorem \ref{MainTheorem} since by Proposition \ref{cptCUsubmanifoldincom}, we have incompressibility among all such compact $cu$-submanifolds hence Theorem \ref{MainTheoremDim4} is a corollary of Theorem \ref{MainTheorem} in the $\mathbb T^4$ case. 
\begin{theorem}( Haefliger argument \cite{Haefliger1962} )
\label{Haefliger}
    On a simply connected manifold $\tilde M$, let $\mathcal F$ be a codimension one foliation and let $\mathcal F^{\perp}$ be a transverse foliation. Then if there is a leaf of $\mathcal F^{\perp} $ that can intersect a leaf of $\mathcal{F}$ more than once, $\mathcal F$ must have a non-simply connected leaf. 
\end{theorem}

\subsection{Proof of Theorem \ref{MainTheorem}}

Let $f:\mathbb T^{n+1}\to \mathbb T^{n+1}$ be a partially hyperbolic diffeomorphism with $\dim E^s=\dim E^c=1$. Now suppose that there is an $f$-invariant $cu$-foliation $\mathcal W^{cu}$ having a compact incompressible leaf $T$. Still by Theorem \ref{AndyCptCuSubmanifolds}, and incompressibility, there is $k\in \N_+$ such that $f^k(T)=T$ and $f^k(\mathbb T^{n+1}\setminus T)=\mathbb T^{n+1}\setminus T\simeq \mathbb T^n\times (0,1)$. 

Then define $\varphi:\mathbb T^n\times [0,1]\to \mathbb T^n\times [0,1]$ such that:
$$\varphi (x):=\begin{cases} x, \text{if } x\in \mathbb T^n\times (0,1)\\(u,1),\text{if }  x=(u,0),u\in \mathbb T^n\\ (u,0),\text{if } x=(u,1),u\in \mathbb{T}^n\end{cases}$$
Then take $\hat f:\mathbb T^n\times [0,1]\to \mathbb T^n\times [0,1]$ such that $\hat f/\varphi=f^{2k}$ ( Note that $\mathbb T^n\times [0,1]/\varphi=\mathbb T^{n+1}$ ). So both $T_0:=\mathbb T^n\times \{0\},T_1:=\mathbb T^n\times \{1\}$ are invariant under $\hat f$ and $\hat f $ is isotopic to $\hat N:=A\times \text{Id}$ where $A$ is an $n\times n$ hyperbolic integer matrix with exactly one-dimensional stable distribution ( by Remark \ref{OneDimensionalStable} ) and $\text{Id}$ is the identity map on $[0,1]$. Set the natural product space projection as $P: \mathbb T^n\times [0,1]\to \mathbb T^n$. 
Then on the corresponding fundamental groups, we have:
$$P_*\circ \hat N=A\circ P_*
$$
Hence by \cite{Franks1970}, there is a semi-conjugacy $h:\mathbb T^n\times [0,1]\to \mathbb T^n$ homotopic to $P$ such that $h\circ \hat f=A_T\circ h$, where $A_{T}:\mathbb T^n\to \mathbb T^n$ is the induced hyperbolic toral automorphism from $A$ (We have $\pi_{T}\circ A=A_{T}\circ \pi_{T}$, where $\pi_T:\mathbb R^n\to \mathbb T^n$ is the natural covering map). Define $k:=h|_{\mathbb T^n\times \{0\}}$, and let $\tilde k:\mathbb R^n\to \mathbb R^n$ be a lift of $k$ and $\tilde{h}:\mathbb{R}^n\times [0,1]\to \mathbb{R}^n\times \{0\}$ be a lift of $h$ such that $\tilde{k}=\tilde h|_{\mathbb{R}^n\times \{0\}}$. 
The lemma below directly follows from the fact that $\tilde h$ has a bounded distance from $P$:  
\begin{lemma}
    \label{hmapsStoS}

$h$ maps stable manifolds of $\hat f$ to stable manifolds of $A_T$.
\end{lemma}
Now take a lift of $\hat f$ as $\hat F$ and set $\pi_{T_0}:\mathbb{R}^n\times \{0\}\to T_0$ as the universal covering map. Like the map $H$ in Section \ref{CodimensionOnePHD}, we have $H_0$ such that $H_0\circ {\hat F}|_{\mathbb{R}^n\times \{0\}}=A\circ H_0$ and the corresponding $h_0:T_0\to T_0$ such that $h_0\circ \pi_{T_0}=\pi_{T_0}\circ H_0$. Also define $L:\mathbb{R}^{n}\times \{0\}\to \mathbb{R}^n$ such that $\forall t\in \mathbb{R}^n,L(t,0)=t$ and the map $l:T_0\to \mathbb{T}^n$ such that $\forall u\in \mathbb{T}^n, l(u,0)=u$. Denote $h_l:=l^{-1}\circ k$ and $H_L:=L^{-1}\circ {\tilde{k}}$. Since $\pi_{T_0}\circ H_L=h_l\circ \pi_{T_0}$, $H_L$ is a lift of $h_l$. 

Since $h_l\circ \hat f|_{T_0}=A_{T_0}\circ h_l$, and the fact that $A_{T_0}$ is a $\pi_1$-diffeomorphism ( see \cite{Franks1970} ), we have $h_l=h_0$. Hence $k=l\circ h_0$. 

The following proposition reveals the relation between $h^{-1}(h(p))$ and $\mathcal W^{cs}_{\text{loc}}(p)$ for points in $\mathbb T^n\times (0,1)$ that are close enough to $T_0$.
\begin{proposition}
    \label{fiberclosetoT0}

$\forall \epsilon>0$ there is a periodic point $p\in T_0$ such that the arc $\gamma _p:={k}^{-1}(k(p))$ has length smaller than $\epsilon$. Let $U$ be a small neighborhood of $\gamma _p$, then $U\cap h^{-1}(h(p))\subset {\mathcal W}^{cs}_{\text{loc}}(p)$. Also, let $\gamma$  be an open center arc strictly containing $\gamma _p:={k}^{-1}(k(p))$. Then for any center arc $\gamma '\subset \mathcal W^{cs}_{\text{loc}}(p)$ close enough to $\gamma$ , we have that $h^{-1}(h(p))\cap \gamma '\ne \emptyset$.
\begin{proof}
    Note that $\forall x\in T_0, {k}^{-1}(k(x))=h_0^{-1}\circ l^{-1}\circ l\circ h_0(x)=h_0^{-1}\circ h_0(x)$, such an arc $\gamma _p$ exists by Proposition \ref{periodicpointwithsmallfiber}. 

Now consider $y\in U\cap h^{-1}(h(p))$. Then set $z$ as the intersection between the local stable manifold of $y$ and $T_0$. Let $\gamma ^s$ be the stable arc connecting $y$ and $z$. Then by Lemma \ref{hmapsStoS}, for $U$ small enough, $h(\gamma ^s)\subset {\mathcal W}^{s}_{\text{loc}}(h(p))$, thus $z\in  k^{-1}(h(\gamma^s))\subset k^{-1}({\mathcal W}^{s}_{\text{loc}}(h(p))) =h_0^{-1}\circ l^{-1}( \mathcal W^s_{\text{loc}}(h(p)))\subset h_0^{-1}(\mathcal{W}^{s}_{\text{loc}}(l^{-1}\circ h(p)))=h_0^{-1}( \mathcal{W}^{s}_{\text{loc}}(h_0(p))\subset  \mathcal W^c(p)$ by Lemma \ref{fibersareCarcs}.  Hence $y\in {\mathcal W}^{cs}_{\text{loc}}(p)$ once $\epsilon>0$ is small enough. 

Also, for the center arc $\gamma '\subset \mathcal W^{cs}_{\text{loc}}(p)$ close enough to $\gamma$  and then for any $y\in \gamma ',$ there is a corresponding $\hat y\in T_0$ which is the intersection between the local stable manifold of $y$ and $T_0$. Obviously, $\hat y\in \gamma$  so we have a center subarc $\gamma ^c\subset \gamma$  joining $\hat y$ and $p$. Then $h(y)\in \mathcal W^s(h(p))$ by Lemma \ref{hgmapsCtoS} and \ref{hmapsStoS} . Hence $h(\gamma ')$ lies in $\mathcal W^s(h(p))$. Note that $h(\gamma )$ is a stable arc containing $h(p)$ in its interior and $h(\gamma )$ is contained in  $\mathcal W^s(h(p))$, which is a line. As a result, by the continuity of $h$, $h(p)\in h(\gamma' )$. So there is $x\in \gamma'$, such that $h(x)=h(p)$, hence $x\in h^{-1}(h(p))$. This implies that $h^{-1}(h(p))\cap \gamma '\ne \emptyset$.
\end{proof}
\end{proposition}
Now we can finish the proof. 
   Let $C>0$ be a constant such that $\text{diam}(\tilde h^{-1}(x))<C,\forall x\in T$ and $\epsilon >0$ small enough such that any two points having distance less than $\epsilon$ lie in a trivializing chart of $\tilde {\mathcal W}^{cu}$. Then there is a large $N\in \mathbb N_+$ such that any set with diameter not larger than $C$ has no more than $N-1$ $\epsilon$-separated points. Pick $p\in T_0$ as in Proposition \ref{fiberclosetoT0}. We might as well choose an appropriate lift $\hat F$ of $\hat{f}$ and a lift $\tilde p$ of $p$ such that $\tilde p$  is $\hat{F}$-periodic. Denote the period of $\tilde p$ under $\hat F$ as $m\in \N_+$. 

It is not difficult to check that properties of $p$ in Proposition \ref{fiberclosetoT0} still hold for $\tilde p$ in the universal cover since it is a local property. So we can take $x_1,...,x_N\in {\tilde {\mathcal W}}^{cs}_{\text{loc}}(\tilde p)\cap {\tilde h}^{-1}(\tilde h(\tilde p))$ such that each of them are in different center curves of $\hat F$. Then ${\hat F}^n({\tilde h}^{-1}(\tilde h(\tilde p)))={\tilde h}^{-1}(\tilde h({\hat F}^n(\tilde p))),\forall n\in \mathbb Z$, we have that $\forall n\in \Z,\text{diam}(\{{\hat F}^n(x_j)\}_{1\le j\le N})<C$. So we can find two points $x_i,x_j$ and a decreasing subsequence $n_k\to -\infty$, such that $m\mid n_k$ and $d({\hat F}^{n_k}(x_i),{\hat F}^{n_k}(x_j))<\epsilon,\forall k\in \N_+$.  Then there are two arcs, $\alpha _1,\alpha _2$ joining $x_i$ with $x_j$, where $\alpha _1$ starts at $x_i$, tangent to $E^c$ and $\alpha _2$ starts at the other endpoint of $\alpha _2$ and ends at $x_j$, tangent to $E^s$. Then since ${\hat F}^{n_k}(\alpha _2)$ will be very long as $n_k\to -\infty$ and ${\hat F}^{n_k}(\alpha _1)$ is contained in a leaf of $\tilde {\mathcal W}^{cu}$. Hence by taking a small perturbation, we have a closed curve $\tilde \gamma \subset \mathbb{R}^n\times [0,1]$ transverse to $\tilde {\mathcal W}^{cu}$. Then a slight modification of Theorem \ref{Haefliger} implies that there is a non-simply connected leaf in $\tilde {\mathcal W}^{cu}$.

\newpage
\section{On Incompressibility of Anosov Tori}
\label{OnIncomAnoTori}
 \begin{proposition}
Let $N$ be a 3-manifold and $\mathbb{S}^2 \neq S \subset N$ is an embedded 2-sided closed surface then the following three statements are equivalent
\begin{enumerate}[label=(\alph*)]
    \item $S$ is incompressible.
    \item There is no embedded disc $D \subset N$ such that $D \cap S = \partial D$ and $\partial D$ is essential in $S$.
    \item For each disc $D \subset N$ with $D \cap S = \partial D$, there is a disc $D' \subset S$ with $\partial D = \partial D'$.
\end{enumerate}
\end{proposition}
\begin{remark}
    Actually, $c$ is the original definition of incompressibility of embedded 2-sided closed surfaces ( see, for instance, \cite{Hatcher2000} ).
\end{remark}
In \cite{RodriguezHertz2008}, the authors gave the proof on the incompressibility of Anosov 2-tori in 3-manifolds:
\begin{theorem} (Theorem 4.1 of \cite{RodriguezHertz2008})
   \label{Ano2ToriIncom}
Anosov 2-torus in an orientable closed 3-manifold are incompressible.
\end{theorem}
The proof of Theorem \ref{Ano2ToriIncom} is a 3-manifold topology argument, the attempt to extend this proof to all dimensions fails at the beginning, which is, after splitting along the torus, at least one of the components does not have an incompressible boundary. This follows from Kneser's Lemma in 3-manifolds while in higher-dimensions there are few higher-dimensional analogues of the Loop Theorem along with its corollaries Kneser's Lemma and Dehn’s Lemma (See for instance, the 4-dimensional Disk Embedding Theorem and some 4-dimensioan analogues of Dehn’s Lemma ), and they are not as satisfactory as those results in 3-manifold topology. As a result, we have to find some methods more algebraic and add some restrictions on the ambient manifold of Anosov tori. 

Also, Anosov 2-tori and 3-tori are, algebraicly special, in the sense that the characteristic polynomials of $2\times 2$ and $3\times 3$ hyperbolic integer matrices are irreducible. However, in higher dimensions, we have hyperbolic matrices like
$$A_0=\begin{bmatrix}
2 & 1 & 0 & 0 \\
1 & 1 & 0 & 0 \\
0 & 0 & 2 & 1 \\
0 & 0 & 1 & 1
\end{bmatrix}$$
which is not irreducible. Now we start the proof of Theorem \ref{IrreAnosovTori}. 

\begin{proof}[Proof of Theorem \ref{IrreAnosovTori}]
    Since $T$ is an embedded irreducible Anosov $n$-tori of $M$, then there is a diffeomorphism $f:M\to M$ such that $f(T)=T$ and $(f|_T)_*:\pi_1(T)\to \pi_1(T)$ induces a hyperbolic integer matrix $A:\mathbb{R}^n\to \mathbb{R}^n$ with irreducible characteristic polynomial $P_A$. 
    Now suppose that $T$ is not incompressible.\\
 \\Claim 1:  $T$ is totally compressible. 
\begin{proof}[Proof of claim]
Denote the universal covering map of $T$ as $\pi_T:\mathbb R^n\to T$, and we will identify $\mathbb Z^n\subset \mathbb R^n$ with $\pi_1(T)$ in the following arguments. 

Suppose there is a nontrivial element $v\in \ker i_*$. By the fact $i\circ f|_T=f\circ i$, $\ker i_*$ is $A$-invariant. It is not difficult to see that $Av, v$ are non-parallel and are both in $\ker i_*$. If, as an induction hypothesis, there is $2\le k\le n-1$ such that $v,Av,...,A^{k-1}v$ are linearly independent whereas $v,Av,...,A^kv$ are not, there are constants $\lambda _j,0\le j\le k$ such that $A^{k}v=\Sigma _{j=0}^{k-1}\lambda _jA^jv$, then for any $w\in W_{k}:= \text{Span}_{\R}\{v,Av,...,A^{k-1}v\}$, let $w=\Sigma _{j=0}^{k-1}c _jA^jv$ for some constants $c _j,0\le j\le k$, we have that $A(w)=\Sigma _{j=0}^{k-1}c _jA^{j+1}v=c_{k-1}A^kv+\Sigma _{j=1}^{k-1}c _{j-1}A^{j}v=c_{k-1}\Sigma _{j=0}^{k-1}\lambda _jA^jv+\Sigma _{j=1}^{k-1}c _{j-1}A^{j}v\in W_k$. It follows that  $W_{k}$ is an $A$-invariant $k$-dimensional vector subspace of $\mathbb R^n$, spanned by $k$ integer-valued vectors. The main result of \cite{Bombieri1983} implies that one can take $v_k,...,v_{n-1}$ as $n-k$ integer-valued vectors in $\mathbb R^n$, spanning the subspace $V_{n-k}$, such that $\mathbb R^n=W_k\oplus V_{n-k}$. Now denote $w_j:=A^jv$, and the integer coeffiients $w_j=(w_{i,j})_{0\le i\le n-1},0\le j\le k-1,$ $v_j=(v_{i,j})_{0\le i\le n-1},k\le j\le n-1$. Define $Q:=(u_{i,j})_{0\le i,j\le n-1}$, where $u_j:=(u_{i,j})_{0\le i\le n-1}=\begin{cases}w_j,0\le j\le k-1\\v_j,k\le j\le n-1\end{cases}$. Then the matrix ( with rational coefficients ) $B:=Q^{-1}AQ$ has the following block structure: 

$$
B=\begin{bmatrix}
B_1 & B_2 \\
0 & B_3
\end{bmatrix}
$$

This implies that the characteristic polynomial $P_A(t)=P_{B}(t)=P_{B_1}(t)P_{B_2}(t)$ admits a non-trivial factrization, contradicting to the irreducibility of $A$. 

By the induction above, $v,Av,...,A^{n-1}v$ are linearly independent integer-valued vectors, then by taking integer multiplications and subtractions ( both of them are invariant operations for subgroups of $\mathbb Z^n$ ), we can find $m_0,...,m_{n-1}\in \Z$ such that $(m_0,0,\dots,0),(0,m_1,\dots,0),\dots,(0,0\dots m_{n-1})\in \ker i_*$. It follows that $\ker i_*=\pi_1(T)$ since $\pi_1(M)$ is torsion-free.
\end{proof}
Denote the universal covering map of $M$ as $\pi_M:\mathbb R^{n+1}\to M$, notice that $i_*(\pi_1(T))=0\le  \pi_M(\mathbb R^{n+1})$, by the lifting criterion we have a lifting $\tilde i$, which is an embedding since $\pi_M\circ \tilde i=i$. Set $\tilde T:=\tilde i\circ i^{-1}(T)$ and we can take a lift $\tilde f :\mathbb R^{n+1}\to\mathbb R^{n+1}$ of $f$, such that $\tilde f(\tilde T)=\tilde T$. By adding a point of infinity $\infty$ and identifying $\mathbb S^{n+1}$ with $\mathbb R^{n+1}\sqcup \{\infty\}$, we have that $\tilde T$ is embedded in $\mathbb S^{n+1}$.

By the tubular neighborhood theorem, we can split $\mathbb S^{n+1}$ by removing $U(\tilde T)$, a small tubular neighborhood of $\tilde T$  which is a manifold with boundary and is homotopy equivalent to $\tilde T$, denote $\hat  S^{n+1}:=\mathbb S^{n+1}-U(\tilde T)$.\\
\\Claim 2: There are exactly two components in $\hat  S^{n+1}$, denoted as $\hat  S^{n+1}_1$ and $\hat  S^{n+1}_2$. 
\begin{proof}[Proof of Claim]
    Actually, \begin{align*}
        H_0(\hat  S^{n+1}))=\tilde H_0(\hat  S^{n+1})\oplus \mathbb{Z}\xlongequal{\text{Alexander duality}}\tilde H^n(U(\tilde T))\oplus \mathbb Z\\\simeq H^n(\tilde T)\oplus \mathbb Z\xlongequal{\text{Poincaré duality}}H_0(\tilde{T})\oplus \mathbb Z=\mathbb Z^2
    \end{align*}, so the number of components of $\hat  S^{n+1}$ are exactly two. 
    
\end{proof}
It follows that $\partial \hat  S^{n+1}$ is a disjoint union of two copies of $\tilde T$, which are boundaries of the components, denoted as $\tilde T_r=\partial \hat   S^{n+1}_r$ with the inclusion map $i_r:\tilde T_r \hookrightarrow \hat  S^{n+1}_r,r=1,2$. Then there is a naturally induced $\hat f:\hat  S^{n+1}\to \hat  S^{n+1}$ such that: 

$$
\begin{cases}\hat f(\tilde T_r)=\tilde T_r,r=1,2\\ \hat f(\infty)=\infty \\\hat f|_{\tilde T_1},\hat  f|_{T_2},\tilde f|_{\tilde T} \text{ are topologically conjugate}\end{cases}
$$

Now denote $U:=\hat  S^{n+1}_1\cup U(\tilde T)$ and $V:=\hat  S^{n+1}_2\cup U(\tilde T)$, obviously they are open subsets of $\mathbb S^{n+1}$ such that $U\cap V=U(\tilde T)$ and $U\cup V=\mathbb S^{n+1}$.

 Denote $i_U :U(\tilde T) \hookrightarrow U$ and $i_V:U(\tilde T) \hookrightarrow  V$. Obviously $G_1:=\ker {i_U}_*\simeq \ker {i_1}_*$ and $G_2:=\ker {i_V}_*\simeq \ker {i_2}_*$.

Following from the same argument as claim 1 applied to $\hat f|_{\tilde T_1},\hat f|_{\tilde T_2}$ topologically conjugate to $\hat f|_{\tilde T}$, we have that $G_j$ is either $0$ or isomorphic to $\mathbb Z^n$, $j=1,2$. So there will be only four possible cases for the pair $(G_1, G_2)$.\\
\\ Claim 3 At least one of $G_1,G_2$ is isomorphic to $\mathbb Z^n$.

\begin{proof}[Proof of Claim]
    In light of Seifert–Van Kampen theorem, for $S:=\{({i_U}_*w )({i_V}_*w )^{-1}| w\in \pi_1(U(\tilde T))\}$ and let $\overline S$ be the normal closure of $C$ in $\pi_1(U)*\pi_1(V)$, we have:
    $$0=\pi_1(\mathbb S^{n+1})=\pi_1(U)*\pi_1(V)\mathbin{/}\overline S$$
    Note that:
    \begin{align*}
\overline{S} 
&= \left\{ \prod_{j=1}^k g_j^{-1} s_j^{\epsilon_j} g_j \ \middle| \ 
k \geq 0, \ \epsilon_j = \pm 1, \ s_j \in S, \ g_j \in \pi_1(U) * \pi_1(V) \right\} \\
&= \left\{ \prod_{j=1}^k 
\left[ \left( \prod_{l=1}^{n_j} g_{j,l} h_{j,l} \right)^{-1} s_j^{\epsilon_j} 
\left( \prod_{l=1}^{n_j} g_{j,l} h_{j,l} \right) \right] 
\ \middle| \ 
k \geq 0, \ \epsilon_j = \pm 1, \ s_j \in S, \right. \\
& \qquad \left. g_{j,l} \in \pi_1(U), \ h_{j,l} \in \pi_1(V), \ 1 \leq j \leq k, \ 1 \leq n_j, \ 1 \leq l \leq n_j \right\} \\
&= \left\{ \prod_{j=1}^k 
\left[ \left( \prod_{l=1}^{n_j} g_{j,l} h_{j,l} \right)^{-1} 
\left[ \left( ({i_U}_* w_j) ({i_V}_* w_j)^{-1} \right)^{\epsilon_j} \right] 
\left( \prod_{l=1}^{n_j} g_{j,l} h_{j,l} \right) \right] 
\ \middle| \right. \\
& \qquad \left. k \geq 0, \ \epsilon_j = \pm 1, \ w_j \in \pi_1(U(\tilde{T})), \ g_{j,l} \in \pi_1(U), \ h_{j,l} \in \pi_1(V), \right. \\
& \qquad \left. 1 \leq j \leq k, \ 1 \leq n_j, \ 1 \leq l \leq n_j \right\}.
\end{align*}
Since $\overline S=\pi_1(U)*\pi_1(V)$, take $u\in \pi_1(U)$ which is non-trivial, there is $k\ge 1,\epsilon_j=\pm 1,w_j\in \pi_1(U(\tilde T)), g_{j,l}\in \pi_1(U),h_{j,l}\in \pi_1(V),1\le j\le k,1\le n_j,1\le l\le n_j$ such that: 
  \begin{align*}
       u=  \prod_{j=1}^k 
\left[ \left( \prod_{l=1}^{n_j} g_{j,l} h_{j,l} \right)^{-1} 
\left[ \left( ({i_U}_* w_j) ({i_V}_* w_j)^{-1} \right)^{\epsilon_j} \right] 
\left( \prod_{l=1}^{n_j} g_{j,l} h_{j,l} \right) \right] 
\ 
    \end{align*}
    As we can see, if $\ker {i_U}_*=\ker {i_V}_*=\emptyset$, ${i_U}_*w_j$ and ${i_V}_*w_j$ are both non-trivial, and it can only be reduced by multiplying some $h_{j,l}\in \pi_1(V)$, while there must be another $h_{j,l}^{-1}$ at the other side since the long word consists of a concatenation of conjugations. This implies that at least one of ${i_U}_*$ and ${i_U}_*$ is not-injective.
\end{proof}
The next claim is on the other side:\\
\\Claim 4: $G_1,G_2$ cannot both be isomorphic to $\mathbb Z^n$. 
\begin{proof}[Proof of Claim]
    If so, the normal closure $\overline C$  is trivial hence both $U$ and $V$ are simply connected

Firstly, by Alexander duality and the universal coefficient theorem:
\begin{align*}
    H_{n-1}(\hat  S^{n+1}_1)\simeq \tilde H^1(V)\simeq H_1(V)\simeq\text{Ab}(\pi_1(V))\simeq\text{Ab}(\pi_1(\hat  S^{n+1}_2))=0\\H_n(\hat  S^{n+1}_1)\simeq\tilde H_n(\hat  S^{n+1}_1)\simeq \tilde H^0(V) =0
\end{align*}
Secondly, consider the following exact sequence with respect to the homology and relative homology groups of $\hat  S^{n+1}_1$:
$$0=H_n(\hat  S^{n+1}_1)\xrightarrow{}H_{n}(\hat  S^{n+1}_1,\tilde T_1)\xrightarrow{}H_{n-1}(\tilde T_1)\xrightarrow{}H_{n-1}(\hat  S^{n+1}_1)=0$$
Then by exactness of the sequence, we have:
$$H_{n}(\hat  S^{n+1}_1,\tilde T_1)\simeq H_{n-1}(\tilde T_1)\simeq \mathbb Z^n$$
However, by Lefschetz duality, we have the following contradiction:
$$n=\dim H_{n}(\hat  S^{n+1}_1,\tilde T_1)=\dim H^1(\hat S_1^{n+1})=\dim H_1(U)=\dim\text{Ab}(\pi_1(U))=0$$

\end{proof}
Now the only case left is, say, $G_1$ is isomorphic to $\mathbb Z^n$.  and $G_2$ is trivial.  

Let $H:={i_V}_*(\pi_1(U(\tilde T) ))$ we have  $H\le \pi_1(V)$ such that $H\simeq \pi_1(\hat T_1)$ and still in light of Seifert–Van Kampen theorem, we have $0=\pi_1(\mathbb S^4)=\pi_1(U)*\pi_1(V)\mathbin{/}\overline {H}$, where $\overline H$ is the normal closure of $H$ in $\pi_1(U)*\pi_1(V)$. 

It is well-known that $\pi_1(U)*\pi_1(V)\mathbin{/}\overline {\pi_1(V)}\simeq \pi_1(U)$, so we have an isomorphism $\phi:\pi_1(U)\to \pi_1(U)*\pi_1(V)\mathbin{/}\overline {\pi_1(V)}$.

By axiom of choice, there is a map $\psi:\pi_1(U)*\pi_1(V)\mathbin{/}\overline {\pi_1(V)}\to \pi_1(U)*\pi_1(V)$ such that $\forall J\in \pi_1(U)*\pi_1(V)\mathbin{/}\overline {\pi_1(V)},\psi(J)\in J$. Then since $\overline H\le \overline {\pi_1(V)}$, the map $\pi_{\overline H}\circ \psi\circ \phi:\pi_1(U)\to \pi_1(U)*\pi_1(V)\mathbin{/}\overline {H}$ is injective where $\pi_{\overline H}:\pi_1(U)*\pi_1(V)\to \pi_1(U)*\pi_1(V)\mathbin{/}\overline {H}$ is the natural quotient map,  hence $\text{Card}(\pi_1(U))\le \text{Card}(\pi_1(U)*\pi_1(V)\mathbin{/}\overline {H})=1$, so $\pi_1(U)$ is trivial. As a result, $\overline H\simeq \pi_1(V)$. 

On the other hand, in light of the Mayer–Vietoris sequence of the homology groups:
$$\cdot \cdot \cdot \xrightarrow{}H_1(U\cap V)\xrightarrow{} H_1(U)\oplus H_1(V)\xrightarrow{}H_1(\mathbb S^{n+1})\xrightarrow{}\cdot \cdot \cdot $$
We have that $0\simeq H_1(U)\oplus H_1(V)\mathbin{/}H$. Notice that $H_1(U)$ is trivial since $\pi_1(U)$ is, by Hurewicz theorem, we have $\pi_1(V)\mathbin{/}[\pi_1(V),\pi_1(V)]\simeq H_1(V)\simeq H$.\\
\\Claim 5: $[\pi_1(V),\pi_1(V)]\cap H=\{e_{\pi_1(V)}\}$
\begin{proof}[Proof of Claim]
    Let $\pi_V:\pi_1(V)\to \pi_1(V)\mathbin{/}[\pi_1(V),\pi_1(V)]$ be the natural quotient map, we now show that $\pi_V|_H$ is an isomorphism from $H$ to $\pi_1(V)\mathbin{/}[\pi_1(V),\pi_1(V)]$. 

$\forall g\in \pi_1(V),$  since $H$ normally generates $\pi_1(V)$, there is $g_1,...,g_m\in \pi_1(V)$ and $h_1,...,h_m\in H$ such that $g=g_1^{-1}h_1g_1\cdot g_2^{-1}h_2g_2...g_m^{-1}h_mg_m$. 

Notice that:
\begin{align*}
    g &\equiv g_1^{-1} h_1 g_1 \cdot g_2^{-1} h_2 g_2 \cdots g_m^{-1} h_m g_m \\
      &\equiv h_1 \cdot g_2^{-1} h_2 g_2 \cdots g_m^{-1} h_m g_m \\
      &\equiv h_1 \cdot h_2 \cdots g_m^{-1} h_m g_m \\
      &\equiv h_1 \cdot h_2 \cdots h_m \quad \big(\text{mod } [\pi_1(V), \pi_1(V)]\big)
\end{align*}
This implies that $\pi_V(g)=\pi_V( h_1\cdot h_2\cdot \cdot \cdot h_m   )\in \pi_V(H)$, hence $\pi_U|_H$ is surjective. Note that $H\simeq  \pi_1(V)\mathbin{/}[\pi_1(V),\pi_1(V)]$, which is free abelian of finite rank. This implies that $\pi_V|_H$ is an isomorphism. 

As a result, $\{e_{\pi_1(V)}\}=\ker\pi_V|_H=\ker \pi_V\cap H=[\pi_1(V),\pi_1(V)]\cap H$.
\end{proof}
Now let $\alpha$  be an action of $\pi_1(U(\tilde T) )$ on $[\pi_1(V),\pi_1(V)]$ such that: $$\forall v\in\pi_1(U(\tilde T) ),g\in [\pi_1(V),\pi_1(V)],\alpha _v(g):={i_V}_*(v)\cdot g\cdot {i_V}_*(v^{-1})$$, then we have a semidirect product $G_V:=[\pi_1(V),\pi_1(V)]\rtimes \pi_1(U(\tilde T)$  with respect to $\alpha$ . 

Set $\Psi: G_V\to \pi_1(V)$ such that $\forall g\in [\pi_1(V),\pi_1(V)]$ and $v\in\pi_1(U(\tilde T) )$ we have $\Psi (g,h):=g{i_V}_*(v)$. We need to show that $\Psi$  is an isomorphism. 

Note that for any $g_1,g_2\in [\pi_1(V),\pi_1(V)]$ and $v_1,v_2\in \pi_1(U(\tilde T) )$ we have that:
\begin{align*}
    \Psi(g_1,h_1) \cdot \Psi(g_2,h_2) 
    &= g_1 \cdot {i_V}_*(v_1) \cdot g_2 \cdot {i_V}_*(v_2) \\
    &= g_1 \cdot {i_V}_*(v_1) \cdot g_2 \cdot {i_V}_*(v_1^{-1}) \cdot {i_V}_*(v_1) \cdot {i_V}_*(v_2) \\
    &= g_1 \alpha_{h_1}(g_2) \cdot {i_V}_*(v_1) {i_V}_*(v_2) \\
    &= \Psi(g_1 \alpha_{h_1}(g_2), {i_V}_*(v_1) {i_V}_*(v_2)) \\
    &= \Psi \big((g_1, {i_V}_*(v_1)) (g_2, {i_V}_*(v_2))\big)
\end{align*}
So $\Psi $ is a homomorphism. Also,  if $\Psi (g,v)=1$ we have that $g={i_V}_*(v^{-1})\in H$. As a result $g\in[\pi_1(V),\pi_1(V)]\cap H$ therefore $g=e_{\pi_1(V)}$ by claim 5.

thus $\Psi$  is injective and the surjectivity of $\Psi$  follows from the surjectivity of $\pi_V|_H$ shown in the proof of claim 5. Hence $\Psi$  is an isomorphism and $\pi_1(V)\simeq G_V$.

Finally, since $H$ normally generates $\pi_1(V)$, $\pi_1(U(\tilde T) )=\Psi ^{-1}(H)$ normally generates $G_V$.

\end{proof}

\begin{remark}
    It seems that arguments in the proof of claim 3 can directly imply both $G_1$ and $G_2$ are isomorphic to $\mathbb Z^n$, but it is not true. For example, take a non-trivial knot $K\subset \mathbb S^3$ and set $U_K\simeq \mathbb D^2\times \mathbb S^1$ as its tubular neighborhood. Denote $G_K:=\pi_1(\mathbb{S}^3-U_K)$ as the knot group and $T_K:=\partial U_K\simeq \mathbb T^2$ is the boundary torus. 

Set $i_K:T_K\hookrightarrow \mathbb{S}^3-U_K,j_K:T_K\hookrightarrow \overline{U_K}$ as the inclusion maps towards the two directions. Then actually  $\ker {i_K}_*=0$ hence $\mathbb{S}^3-U_K$ has an incompressible torus boundary. Set embedded circles,$\gamma _l,\gamma _m$ of the classes as two generators of $\pi_1(T_K)$ respectively, such that $i_K(\gamma _l)$ is of the longitude of $G_K$  and an isotopy of $\gamma _m$  is of the meridian of $G_K$. 

On the other hand, Seifert–Van Kampen Theorem gives $\overline {C_K}=\pi_1(\overline {U_K})*\pi_1(\mathbb{S}^3-U_K)$ where $\overline {C_K}$ is the normal closure of $C_K:=\{({i_K}_*w)({j_K}_*w)^{-1}|w\in \pi_1(T_K)\}$ in $\pi_1(\overline {U_K})*\pi_1(\mathbb{S}^3-U_K)$. Then we must have $\pi_1(\overline {U_K}),G_K\le \overline {C_K}$, so we need to find elements in $\overline {C_K}$ to represent ${j_K}_*([\gamma ]_l)$

Since a meridian is a typical weight element of $G_K$ , ${i_K}_*[\gamma _m]$ normally generates $G_K$, thus we have $g_1,..,g_k\in G_K,\epsilon_1,...,\epsilon_k=\pm1,$   such that:

$$
{i_K}_*([\gamma _l])= \prod_{j=1}^k g_j^{-1}({i_K}_*[\gamma _m])^{\epsilon_j}g_j
$$

Then we have 

$$({j_K}_*[\gamma _l])=({j_K}_*[\gamma _l])({i_K}_*[\gamma _l])^{-1}({i_K}_*[\gamma _l])=({j_K}_*[\gamma _l])({i_K}_*[\gamma _l])^{-1} \prod_{j=1}^k g_j^{-1}({i_K}_*[\gamma _m])^{\epsilon_j}g_j\in \overline {C_K}$$.
\end{remark}

\end{document}